\documentclass[12pt,reqno]{amsart}
\usepackage{tikz-cd}
\usepackage{url}
\usepackage{amsfonts,amsmath,mathtools}

	\newtheorem{Theorem}{Theorem}[section]
	\newtheorem{Lemma}[Theorem]{Lemma}
	\newtheorem{Corollary}[Theorem]{Corollary}
	\newtheorem{Proposition}[Theorem]{Proposition}
	\newtheorem{Remark}[Theorem]{Remark}
	
	\newtheorem{Example}[Theorem]{Example}
	\newtheorem{Examples}[Theorem]{Examples}

	\newtheorem{Conjecture}[Theorem]{Conjecture}
	\newtheorem{Question}[Theorem]{Question}
	
	\def\qed{\ifhmode\textqed\fi
		\ifmmode\ifinner\hfill\quad\qedsymbol\else\dispqed\fi\fi}
	\def\textqed{\unskip\nobreak\penalty50
		\hskip2em\hbox{}\nobreak\hfill\qedsymbol
		\parfillskip=0pt \finalhyphendemerits=0}
	\def\dispqed{\rlap{\qquad\qedsymbol}}
	
	\def\mm{\mathfrak{m}}
	\def\m{\mathfrak{m}}
	\def\nn{\mathfrak{n}}

	\def\depth{\textup{depth}}
	\def\reg{\textup{reg}}
	\def\gr{\textup{gr}}
	
	\def\ZZ{\mathbb{Z}}
	\def\supp{\textup{supp}}
	
\begin{document}
	
	\title{Dao numbers and\\ the asymptotic Behaviour of fullness}
	\author{Antonino Ficarra}
	
	\address{Antonino Ficarra, Department of mathematics and computer sciences, physics and earth sciences, University of Messina, Viale Ferdinando Stagno d'Alcontres 31, 98166 Messina, Italy}
	\email{antficarra@unime.it}
	
	\subjclass[2020]{Primary 13C15, 05E40, 05C70}
	\keywords{colon ideals, $\mm$-full ideals, Koszul algebras.}
	\date{}

	\begin{abstract}
		In the present paper, we study the Dao numbers $\mathfrak{d}_1(I),\mathfrak{d}_2(I)$ and $\mathfrak{d}_3(I)$ of an ideal $I$ of a Noetherian local ring $(R,\mathfrak{m},K)$ or a standard graded Noetherian $K$-algebra. They are defined as the smallest $\ell\ge0$ such that $I\mathfrak{m}^k$ is $\mathfrak{m}$-full, full, weakly $\mathfrak{m}$-full, respectively, for all $k\ge\ell$. We provide general bounds for the Dao numbers in terms of the Castelnuovo--Mumford regularity of certain modules over the Rees algebra $\mathcal{R}(\mathfrak{m})$. If $R$ is a Koszul algebra, we prove that the Dao numbers are less or equal to $\text{reg}_{\text{gr}_\mathfrak{m}(R)}\text{gr}_\mathfrak{m}(I)$, where $\text{gr}_\mathfrak{m}(I)$ is the associated graded module of $I$. Finally, for monomial ideals, we combinatorially bound the Dao numbers in terms of asymptotic linear quotients and bounding multidegrees.
	\end{abstract}
	
	\maketitle
	
	\section*{Introduction}
	
	Let $(R,\mm,K)$ be either a Noetherian local ring or a standard graded Noetherian $K$-algebra with unique graded maximal ideal $\mm$. Hereafter, we will always assume that $K$ is an infinite field and that $\depth R>0$. These assumptions will be explained in a moment. Let $I\subset R$ be an ideal, we assume that $I$ is homogeneous if $R$ is a $K$-algebra. The ideal $I$ is called
	\begin{enumerate}
		\item[(a)] \textit{$\mm$-full} if $(I\mm:x)=I$ for a generic element $x\in\mm-\mm^2$,
		\item[(b)] \textit{full} if $(I:x)=(I:\mm)$ for a generic element $x\in\mm-\mm^2$,
		\item[(c)] \textit{weakly $\mm$-full} if $(I\mm:\mm)=I$.
	\end{enumerate}

    Here, that $x\in\mm-\mm^2$ is \textit{generic} means that $x$ belongs to some non-empty open set of the Zariski topology of $R$.
    
    Let $P$ be one of the properties $\{\mm\textit{-full,\ full,\ weakly}\ \mm\textit{-full}\}$ and let $I,J\subset R$ be two ideals. In \cite[Theorem 3.1]{Dao21}, Hailong Dao proved that $J$ has the property $P$ if and only if $J+I\mm^k$ has the property $P$ for all $k\gg0$ (the result is stated in the local case, but the proof carries over in the graded case). In particular, there exists an ideal $I\subset R$ for which $I\mm^k$ is (weakly) $\mm$-full for all $k\gg0$ if and only if $\depth R>0$ \cite[Corollary 3.2]{Dao21}. Thus, under our assumptions, for all $k\gg0$, $I\mm^k$ has any of the properties $P$ considered above. Based on this fact, Dao introduced the following numerical invariants \cite[Definition 3.1]{Dao21}, which we call the \textit{Dao numbers} of $I$:
    \begin{align*}
    	\mathfrak{d}_1(I)\ &=\ \min\{\ell\ge0\ :\ I\mm^k\ \textup{is}\ \mm\textup{-full for all}\ k\ge\ell\},\\
    	\mathfrak{d}_2(I)\ &=\ \min\{\ell\ge0\ :\ I\mm^k\ \textup{is}\ \textup{full for all}\ k\ge\ell\},\\
    	\mathfrak{d}_3(I)\ &=\ \min\{\ell\ge0\ :\ I\mm^k\ \textup{is weakly}\ \mm\textup{-full for all}\ k\ge\ell\}.
    \end{align*}
    
    It is shown in \cite[Proposition 2.2]{MNQ23} that $$\mathfrak{d}_2(I)\ \le\ \mathfrak{d}_3(I)\ =\ \mathfrak{d}_1(I)$$ if $K$ is infinite and $\depth R>0$. This fact justifies these recurring assumptions we imposed in the beginning. In \cite{Dao21}, Dao raised the problem of finding good bounds for the Dao numbers. For reduction ideals of the maximal ideal in a local ring, this question has been answered by Miranda-Neto and Queiroz \cite{MNQ23}.
    
    The concept of $\mm$-fullness was first introduced by David Rees in unpublished work, and later developed by Junzo Watanabe \cite{W87,W91,W91b}. Recently, fullness of ideals was considered by many reaserchers \cite{CK21,Dao21,DaoK20,FMNQ,HW12,HW15,HRR2002,MN2017,Rush13}. Surprisingly, these concepts are also related to the classical Zariski-Lipman conjecture (about derivations and smoothness) in the open case of surfaces \cite[Section 4]{MN2017}.
    
    It is straightforward to see that a $\mm$-full ideal is weakly $\mm$-full. A \textit{reduction} of $\mm$ is an ideal $I$ such that $I\mm^k=\mm^{k+1}$ for some $k$. If $R$ is regular (with infinite residue field $R/\mm$) and $I$ is a reduction of $\mm$, then $I$ is $\mm$-full \cite[Corollary 3.11]{MNQ23}. If $R/\mm$ is infinite, then integrally closed ideals are $\mm$-full \cite[Theorem (2.4)]{Goto87}. If $R=K[x_1,\dots,x_n]$ is the polynomial ring with coefficients over an infinite field $K$, then any componentwise linear ideal is $\mm$-full \cite{CDNR10} (see also \cite[Proposition 18]{HW12}). We extend this result to any Koszul algebra (Corollary \ref{Cor:CLIDao}). In \cite{DaoK20} \textit{Burch ideals} are introduced. An ideal $I$ is called \textit{Burch} if $(I\mm:\mm)\ne(I:\mm)$. If $\depth R/I=0$ and $I$ is weakly $\mm$-full, then $I$ is Burch \cite[Corollary 2.4]{DaoK20}.
    
    The next scheme summarizes the relationship between the concepts discussed.
    
    \[\begin{matrix}
    	{\textup{Componentwise linear}}\,\,\,\,\,\,\,\,\,\,\,\,\,\,\,\,\,\,\,\,\,\,\,\,\,\,\,\,\,\,\,\,\\[6pt]
    	\,\,\,\,\,\,\rotatebox{90}{$\Longleftarrow$}\,\,\substack{\big(\\ \phantom{ll}\\ \phantom{..}}\substack{\small\textup{if}\ R\ \textup{is a}\\ \small \textup{Koszul algebra}\\ \phantom{aaaa}\\ \phantom{..}}\substack{\big)\\ \phantom{ll}\\ \phantom{..}}\\[-10pt]
    	\textup{Reduction of}\ \mm\,\,\,\xRightarrow{\substack{(\small R\ 	\textup{regular})\\ \phantom{...}}}\,\,\,{\mathfrak{m}\textup{-full}}\,\,\Longrightarrow\,\,\textup{Weakly}\ \mathfrak{m}\textup{-full}\,\, \xRightarrow{\substack{(\depth R/I=0)\\ \phantom{aa}}}\,\,\textup{Burch} \\[7pt]
    	\,\,\!\,\rotatebox{90}{$\Longrightarrow$}\,\,\substack{\big(\\ \phantom{l}\\ \phantom{.}}\substack{{\small\textup{if}\ R/\mm\ \textup{is an}}\\ \small\textup{infinite field}\\[6pt] \phantom{aaa}}\substack{\big)\\ \phantom{l}\\ \phantom{.}}\\[-3pt]
    	{\textup{Integrally closed}}\,\,\,\,\,\,\,\,\,\,\,\,\,\,\,\,\,\,\,\,\,\,\,\,\,\,\,\,\,\,\,\,\,\,
    \end{matrix}\]\medskip
    
    In the present paper, we determine general bounds for the Dao numbers in both the local and graded settings. Moreover, we will consider Koszul algebras and monomial ideals, and provide more specific bounds for the Dao numbers in such cases.
    
    The paper proceeds as follows. In Section \ref{Sec:Dao1} we address Dao's question and we bound the Dao numbers of any ideal $I$. For this aim, we introduce the \textit{Dao module} $\mathfrak{D}_\mm(I)$ of $I$ which is defined as $\bigoplus_{k\ge0}(I\mm^{k+1}:\mm)/(I\mm^k)$ and has the structure of a module over the Rees algebra $\mathcal{R}(\mm)=\bigoplus_{k\ge0}\mm^k$ of $\mm$. The $k$th component of this module is zero if and only if $I\mm^k$ is weakly $\mm$-full. Since $\mathfrak{D}_\mm(I)_k=0$ for all $k\ge\mathfrak{d}_3(I)=\mathfrak{d}_1(I)$, it follows that $\mathfrak{D}_\mm(I)$ has finite length, and thus is a finitely generated $\mathcal{R}(\mm)$-module. In Corollary \ref{Cor:regD123-Dao} we note that if $I$ is not weakly $\mm$-full, then $\mathfrak{d}_1(I)=\mathfrak{d}_3(I)=\reg_{\mathcal{R}_\mm(I)}\mathfrak{D}_\mm(I)+1$. In our main Theorem \ref{Thm:DaoInv} we prove that:
    $$
    \mathfrak{d}_2(I)\le \mathfrak{d}_3(I)=\mathfrak{d}_1(I)\ \le \ \max\{\reg_{\mathcal{R}(\mm)}\mathcal{R}(\mm,I),\,\reg_{\mathcal{R}(\mm)}\mathcal{R}(\mm,I)_{\ge1}:_{\mathcal{R}(R)}\mm\}.
    $$
    Here $\mathcal{R}(\mm,I)=\bigoplus_{k\ge0}I\mm^k$ is the extension of the ideal $I$ in the ring $\mathcal{R}(\mm)$, and $\mathcal{R}(\mm,I)_{\ge1}:_{\mathcal{R}(R)}\mm=\bigoplus_{k\ge0}(I\mm^{k+1}:\mm)$. Under the more restrictive assumption that $\depth\gr_\mm(I)>0$ (here $\gr_\mm(I)=\bigoplus_{k\ge0}(I\mm^k/I\mm^{k+1})$ is the associated graded module of $I$), or when $R$ is regular (Theorems \ref{Thm:depth-gr-Dao} and Corollary \ref{Cor:reg-loc-Dao}) then we even have
    $$
    \mathfrak{d}_2(I)\le\mathfrak{d}_3(I)=\mathfrak{d}_1(I)\ \le\ \reg_{\mathcal{R}(\mm)}\mathcal{R}(\mm,I).
    $$
    This bounds holds because $\mathfrak{D}_\mm(I)$ is equal to $(0:_{\mathcal{R}(\mm)/\mathcal{R}(\mm,I)}\mathcal{R}(\mm)_+)$ under the above assumptions (Corollary \ref{Cor:Dao-m-Soc}). Due to these results, we could expect that:
    \begin{Conjecture}
    	For all ideals $I\subset R$:\ \ $\mathfrak{d}_2(I)\le\mathfrak{d}_3(I)=\mathfrak{d}_1(I)\ \le\ \reg_{\mathcal{R}(\mm)}\mathcal{R}(\mm,I)$.
    \end{Conjecture}

    In Corollary \ref{Cor:MNQ} we reobtain the nice result \cite[Corollary 3.11]{MNQ23}, due to Miranda Neto and Queiroz, which says that $\mathfrak{d}_i(I)=0$ for all $i$ if $I$ is a reduction of the maximal ideal $\mm$ and $R$ is regular.
    
    In Section \ref{Sec:Dao2}, we consider the more tame situation of a Koszul algebra. A standard graded $K$-algebra $R$ is \textit{Koszul} if $\reg_RK=0$. A homogeneous ideal $I$ of a Koszul algebra $R$ is called \textit{componentwise linear} if $I_{\langle j\rangle}$ has a linear resolution for all $j$. It turns out that componentwise linear ideals are full, $\mm$-full and weakly $\mm$-full (Corollary \ref{Cor:CLIDao}). In the polynomial ring case, componentwise linear ideals coincide with the so-called \textit{completely $\mm$-full ideals} \cite{HW15}. So the above result is not surprising. To the best of our knowledge, the concept of completely $\mm$-full ideal has not been defined yet in a Koszul algebra. In Theorem \ref{Thm:FHM-Dao}, we notice that for all $k\ge\reg_{\gr_\mm(R)}\gr_\mm(I)$, $I\mm^k$ is componentwise linear. Hence, the number $\reg_{\gr_\mm(R)}\gr_\mm(I)$ is an upper bound for the Dao numbers of $I$. If $R$ is a strongly Koszul algebra, any monomial ideal $I\subset R$ with linear quotients order $\mathcal{O}:u_1<\dots<u_m$ with $\deg(u_1)\le\dots\le\deg(u_m)$ is componentwise linear (Proposition \ref{Prop:LQ-CL-Dao}). If $R$ is the polynomial ring $K[x_1,\dots,x_n]$, then by the Bj\"orner-Wachs rearrangement lemma \cite[Lemma 2.1]{JZ10} (see also \cite{BW96}) the condition $\deg(u_1)\le\dots\le\deg(u_m)$ can be omitted. It is not clear whether this is also the case for any strongly Koszul algebra.
    
    In the final Section \ref{Sec:Dao3}, we consider monomial ideals in the standard graded polynomial ring $S=K[x_1,\dots,x_n]$. Let $I\subset S$ be a monomial ideal with minimal monomial generating set $\mathcal{G}(I)$. It was noted in \cite[Theorem 1.5]{FHM23} (Proposition \ref{Prop:FHM-Dao}) that $I\mm^k$ has linear quotients for all $k\gg0$, which strengthens the fact that $I\mm^k$ is componentwise linear for all $k\gg0$. Analyzing carefully the proof of this fact, which is based on combinatorial arguments, we obtain in Theorem \ref{Thm:LinQuotBoundDao} the combinatorial bound
    \[
    \mathfrak{d}_2(I)\le\mathfrak{d}_3(I)=\mathfrak{d}_1(I)\ \le\ \reg_{\gr_\mm(S)}\gr_\mm(I)\ \le\ \min_{\mathcal{O}}(\max_{u\in\mathcal{G}(I)}\lambda_{I,\mathcal{O},u}).
    \]
    Here the minimum runs through all orders $\mathcal{O}$ of $\mathcal{G}(I)$. In Corollary \ref{Cor:d1d2d3Dao} we prove that $\reg_{\gr_\mm(S)}\gr_\mm(I)\le(\sum_{u\in\mathcal{G}(I)}\deg(u))+1-\mu(I)-\omega(I)$. This bound is sharp. Indeed, equality holds for monomial complete intersections \cite[Theorem 4.1(b)]{ACF3}. Next, in Theorem \ref{Thm:GraphsDao}, we show that $\mathfrak{d}_i(I)\le\max\{|{\bf deg}(I)|-n,0\}$ for $i=1,2,3$, where ${\bf deg}(I)$ is the bounding multidegree of $I$ \cite{F2}. In particular, if $I$ is squarefree, then all Dao numbers are zero. In Corollary \ref{Cor:CIDao} we compute the Dao numbers of monomial complete intersections. This result shows that in general it may be difficult to compute the Dao numbers of a general complete intersection ideal. 
    \section{General bounds for the Dao numbers}\label{Sec:Dao1}
    
    In this first section, we answer Dao's question. For this aim, we need to consider Castelnuovo--Mumford regularity over general base rings.
    
    Hereafter, we follow closely \cite[Section 8.1]{BCRV22}. Let $R=\bigoplus_{k\ge0}R_k$ be a standard graded Noetherian algebra. That is, $R_0$ is a commutative Noetherian ring and $R$ is generated as an $R_0$-algebra by finitely many elements $f_1,\dots,f_n$ of degree one. Let $Q=R_+=\bigoplus_{k>0}R_k$ be the ideal of $R$ generated by the elements of positive degree. Let $M=\bigoplus_{k\ge0}M_k$ be a finitely generated graded $R$-module. The \textit{initial degree} and the \textit{final degree} of $M$, denoted by $\alpha(M)$ and $\omega(M)$, are defined as
    \begin{align*}
    	\alpha(M)\ &=\ \min\{j\ :\ (M/Q M)_j\ne0\}\ =\ \min\{j\ :\ M_j\ne0\},\\
    	\omega(M)\ &=\ \max\{j\ :\ (M/Q M)_j\ne0\}.
    \end{align*}

    The \textit{Castelnuovo--Mumford regularity} of $M$ is defined as \cite[Theorem 8.1.3]{BCRV22}:
    \begin{align*}
    	\reg_R M\ &=\ \max\{j+i\ :\ \textup{H}^i_Q(M)_j\ne0\}\\
    	&=\ \max\{j-i\ :\ H_i({\bf y};M)_j\ne0\}.
    \end{align*}
    Here $\textup{H}_{Q}^i(M)$ is the $i$th local cohomology module of $M$ with support on $Q$, and $H_i({\bf y};M)$ is the $i$th Koszul homology module of ${\bf y}:y_1,\dots,y_n$ with respect to $M$, where ${\bf y}$ is a minimal homogeneous system of generators of $Q$.\medskip
    
    We record some basic properties which we will need in a moment. These properties are also stated in \cite[Page 277, (a), (c) and (d)]{BCRV22} and \cite[page 268]{HH2011}.\smallskip
    \begin{enumerate}
    	\item[(i)] Let $M(j)$ be the module $M$ whose degrees are shifted by $j$: $M(j)_i=M_{i+j}$ for all $i$. Then $\omega(M(j))=\omega(M)-j$ and $\reg_R M(j)=\reg_RM-j$.
    	\item[(ii)] Let $0\rightarrow M\rightarrow N\rightarrow P\rightarrow0$ be a short exact sequence of finitely generated graded $R$-modules. Then,
    	\begin{align*}
    		\reg_R M\ &\le\ \max\{\reg_RN,\,\reg_RP+1\},\\
    		\reg_R P\ &\le\ \max\{\reg_RN,\,\reg_RM-1\}.
    	\end{align*}
    	\item[(iii)] If $M_j=0$ for all $j\gg0$, then $\reg_RM=\max\{j:M_j\ne0\}$.
    	\item[(iv)] \label{fact:page}Let $I\subset R$ be a homogeneous ideal and ${\bf y}:y_1,\dots,y_n$ be a minimal homogeneous system of generators of $Q$. Then
    	\begin{align*}
    		H_n({\bf y};R/I)\ &\cong\ e_1\wedge\cdots\wedge e_n\ \textup{Soc}_Q(R/I)\ =\ e_1\wedge\cdots\wedge e_n\ (0:_{R/I}Q)\ \\&=\ e_1\wedge\cdots\wedge e_n\ (I:Q)/I.
    	\end{align*}
    	Here, $\textup{Soc}_Q(R/I)$ is called the \textit{socle} of $I$ and $e_1,\dots,e_n$ is the basis of a graded free $R$-module of rank $n$, with $\deg e_i=1$ for all $i$. Hence, we have
    	\begin{align*}
    		\max\{j\ :\ \textup{Soc}_Q(R/I)_j\ne0\}\ &=\ \max\{j-n\ :\ H_n({\bf y};R/I)_j\ne0\}\\
    		&\le\ \max\{j-i\ :\ H_i({\bf y};R/I)_j\ne0\}\ \\&=\ \reg_R R/I.
    	\end{align*}
        \item[(v)] Let $I\subset R$ be an homogeneous ideal. Then $\reg_R R/I=\reg_RI-1$. This follows because $H_i({\bf y};R/I)_j\cong H_{i-1}({\bf y};I)_j$ for all $i$, where ${\bf y}$ is a minimal system of generators of $Q$.
    \end{enumerate}\medskip

    Let $M=\bigoplus_{k\ge0}M_k$ be a graded $R$-module and let $\ell\ge0$ be a positive integer. Then $M_{\ge\ell}=\bigoplus_{k\ge\ell}M_k$ is called a truncation of $M$.\medskip

    We are now ready to deliver the promised bound for the Dao numbers.\medskip
    
    Let $R$ be a commutative ring and $I$ be an ideal. The \textit{Rees algebra} of $I$ is the graded ring $\mathcal{R}(I)=\bigoplus_{k\ge0}I^k$. Let $J\subset R$ be another ideal, then $\mathcal{R}(I,J)=\bigoplus_{k\ge0}JI^k$ is the extension of $J\subset R$ in the ring $\mathcal{R}(I)$. In particular, $\mathcal{R}(I,J)$ is a finitely generated ideal of $\mathcal{R}(I)$ if $J$ is a finitely generated ideal of $R$.
    
    \begin{Theorem}\label{Thm:DaoInv}
    	Let $(R,\mm,K)$ be either a Noetherian local ring or a standard graded Noetherian $K$-algebra with unique homogeneous maximal ideal $\mm$. We assume that $K$ is infinite and $\depth R>0$. Let $I\subset R$ be an ideal. We assume that $I$ is homogeneous if $R$ is a $K$-algebra. Then
    	$$
    	\mathfrak{d}_2(I)\le \mathfrak{d}_3(I)=\mathfrak{d}_1(I)\ \le \ \max\{\reg_{\mathcal{R}(\mm)}\mathcal{R}(\mm,I),\,\reg_{\mathcal{R}(\mm)}\mathcal{R}(\mm,I)_{\ge1}:_{\mathcal{R}(R)}\mm\}.
    	$$
    	Here $\mathcal{R}(\mm,I)_{\ge1}:_{\mathcal{R}(R)}\mm=\mathcal{R}(\mm,I)_{\ge1}:_{\mathcal{R}(R)}\mathcal{R}(\mm,\mm)=\bigoplus_{k\ge0}(I\mm^{k+1}:\mm)$.
    \end{Theorem}
    \begin{proof}
    	Notice that for all $k\ge0$ we have $I\mm^{k}\subseteq(I\mm^{k+1}:\mm)$. Therefore, we obtain the short exact sequence
    	\begin{equation}\label{eq:ShortExactDao}
    		0\rightarrow I\mm^{k}\rightarrow (I\mm^{k+1}:\mm)\rightarrow\frac{(I\mm^{k+1}:\mm)}{I\mm^{k}}\rightarrow0.
    	\end{equation}
    	
    	The Rees algebra $\mathcal{R}(\mm)$ of $\mm$ is a standard graded Noetherian $R$-algebra. Since $\mathcal{R}(\mm,I)=\bigoplus_{k\ge0}I\mm^k$ is just the extension of the ideal $I\subset R$ in the ring $\mathcal{R}(\mm)$, it is a finitely generated graded $\mathcal{R}(\mm)$-module. Taking the direct sum of the exact sequences (\ref{eq:ShortExactDao}) for all $k\ge0$, we obtain the short exact sequence:
    	\begin{equation}\label{eq:ShortExactDao1}
    		0\rightarrow\bigoplus_{k\ge0} I\mm^{k}\rightarrow \bigoplus_{k\ge0}(I\mm^{k+1}:\mm)\rightarrow \bigoplus_{k\ge0}\frac{(I\mm^{k+1}:\mm)}{I\mm^{k}}\rightarrow 0.
    	\end{equation}
    	Notice that
    	\begin{equation}\label{eq:colonDao}
    	\begin{aligned}
    		\mathcal{R}(\mm,I)_{\ge1}(1):_{\mathcal{R}(R)}\mm\ \ &=\ \ (\bigoplus_{k\ge0}I\mm^{k+1}):_{\mathcal{R}(R)}\mm\ =\ \bigoplus_{k\ge0}(I\mm^{k+1}:\mm)\\
    		&=\ \ (\mathcal{R}(\mm,I)_{\ge1}:_{\mathcal{R}(R)}\mm)(1).
    	\end{aligned}
    	\end{equation}
        
        We set $\mathfrak{D}_\mm(I)=(\mathcal{R}(\mm,I)_{\ge1}:_{\mathcal{R}(R)}\mm)(1)/\mathcal{R}(\mm,I)$ and notice that for all $k$,
        $$
        \mathfrak{D}_\mm(I)_k\ =\ \frac{(I\mm^{k+1}:\mm)}{I\mm^k}.
        $$
        
        By \cite[Theorem 3.1 and Corollary 3.2]{Dao21}, we have that $I\mm^k$ is weakly $\mm$-full for all $k\ge\mathfrak{d}_3(I)$, that is $(I\mm^{k+1}:\mm)=I\mm^k$. Thus, we have $\mathfrak{D}_\mm(I)_k=0$ for all $k\ge\mathfrak{d}_3(I)$, and if $\mathfrak{d}_3(I)-1\ge0$ then $\mathfrak{D}_\mm(I)_{\mathfrak{d}_3(I)-1}\ne0$. It follows that $\mathfrak{D}_\mm(I)$ is a finitely generated $\mathcal{R}(\mm)$-module, and from (iii)
        $$
        \reg_{\mathcal{R}(\mm)}\mathfrak{D}_\mm(I)\ =\ \max\{\mathfrak{d}_3(I)-1,\,0\}.
        $$
        
        If $\mathfrak{d}_3(I)=0$ there is nothing to prove. Thus, we suppose that $\mathfrak{d}_3(I)>0$. Then $\reg_{\mathcal{R}(\mm)}\mathfrak{D}_\mm(I)=\mathfrak{d}_3(I)-1$. From equations (\ref{eq:ShortExactDao1}) and (\ref{eq:colonDao}) we obtain the following short exact sequence of graded $\mathcal{R}(\mm)$-modules:
        \begin{equation}\label{eq:ShortExactDao2}
        	0\rightarrow\mathcal{R}(\mm,I)\rightarrow(\mathcal{R}(\mm,I)_{\ge1}:_{\mathcal{R}(R)}\mm)(1)\rightarrow\mathfrak{D}_\mm(I)\rightarrow0.
        \end{equation}
        
        Since $\mathcal{R}(\mm,I)$ and $\mathfrak{D}_\mm(I)$ are finitely generated $\mathcal{R}(\mm)$-modules, it follows that $(\mathcal{R}(\mm,I)_{\ge1}:_{\mathcal{R}(R)}\mm)(1)$ is a finitely generated graded $\mathcal{R}(\mm)$-module, as well. Therefore, applying the rules (ii) and (i) to the sequence (\ref{eq:ShortExactDao2}) we obtain that
    	\begin{align*}
    		\reg_{\mathcal{R}(\mm)}\mathfrak{D}_\mm(I)\ &\le\ \max\{\reg_{\mathcal{R}(\mm)}(\mathcal{R}(\mm,I)_{\ge1}:_{\mathcal{R}(R)}\mm)(1),\,\reg_{\mathcal{R}(\mm)}\mathcal{R}(\mm,I)-1\}\\
    		& \le\ \max\{\reg_{\mathcal{R}(\mm)}\mathcal{R}(\mm,I)_{\ge1}:_{\mathcal{R}(R)}\mm,\,\reg_{\mathcal{R}(\mm)}\mathcal{R}(\mm,I)\}-1.
    	\end{align*}
    
        Taking into account that $\reg_{\mathcal{R}(\mm)}\mathfrak{D}_\mm(I)=\mathfrak{d}_3(I)-1$, we obtain the asserted inequality for $\mathfrak{d}_3(I)$. Finally, under the assumptions that $K$ is infinite and $\depth R>0$, it follows from \cite[Proposition 2.2]{MNQ23} that $\mathfrak{d}_2(I)\le\mathfrak{d}_1(I)=\mathfrak{d}_3(I)$.
\end{proof}
    
    We call the $\mathcal{R}(\mm)$-module, considered in the above proof,
    $$
    \mathfrak{D}_\mm(I)\ =\ \bigoplus_{k\ge0}\frac{(I\mm^{k+1}:\mm)}{I\mm^k}
    $$
    the \textit{Dao module} of $I$. As shown in the proof, we have
    \begin{Corollary}\label{Cor:regD123-Dao}
    	Under the same assumptions and notation of Theorem \ref{Thm:DaoInv}, we have
    	$$
    	\mathfrak{d}_2(I)\ \le\ \mathfrak{d}_1(I)=\mathfrak{d}_3(I)\ =\ \begin{cases}
    		\reg_{\mathcal{R}(\mm)}\mathfrak{D}_\mm(I)+1&\textit{if}\ \mathfrak{D}_\mm(I)\ne0,\\
    		\hfill 0&\hfill\textit{otherwise.}
    	\end{cases}
    	$$
    \end{Corollary}

    Recall that the \textit{associated graded ring} of $R$ is $\gr_\mm(R)=\bigoplus_{k\ge0}(\mm^k/\mm^{k+1})$. If $R$ is Noetherian, as we assume in our case, then $\gr_\mm(R)$ is a graded Noetherian ring. We can strengthen Theorem \ref{Thm:DaoInv} as follows.
    \begin{Theorem}\label{Thm:depth-gr-Dao}
    	Under the same assumptions and notation of Theorem \ref{Thm:DaoInv}, if we suppose in addition that $\depth\gr_\mm(R)>0$, then
    	$$
    	\mathfrak{d}_2(I)\le \mathfrak{d}_3(I)=\mathfrak{d}_1(I)\ \le\ \reg_{\mathcal{R}(\mm)}\mathcal{R}(\mm,I).
    	$$
    \end{Theorem}
    \begin{proof}
    	If $\mathfrak{d}_1(I)=\mathfrak{d}_3(I)=0$, there is nothing to prove. Thus, we suppose that $\mathfrak{d}_1(I)=\mathfrak{d}_3(I)>0$. In particular, $\mathfrak{d}_1(I)-1=\mathfrak{d}_3(I)-1=\max\{k:\mathfrak{D}_\mm(I)_k\ne0\}$. Let $Q=\mathcal{R}(\mm)_+=\bigoplus_{k>0}\mm^k$. Notice that
    	\begin{equation}\label{eq:computationDao}
    	\begin{aligned}
    		(\mathcal{R}(\mm,I):_{\mathcal{R}(\mm)}Q)_k\
    		&=\ \{f\in\mathcal{R}(\mm)_k=\mm^k\ :\ f\mm\subseteq I\mm^{k+1}\}\\
    		&=\ \mm^k\cap\{f\in R\ :\ f\mm\subseteq I\mm^{k+1}\}\\
    		&=\ \mm^k\cap(I\mm^{k+1}:\mm).
    	\end{aligned}
    	\end{equation}
        We claim that for all $k\ge0$,
        \begin{equation}\label{eq:claimDao}
        	\mm^k\cap(I\mm^{k+1}:\mm)\ =\ (I\mm^{k+1}:\mm).
        \end{equation}
        The assumptions that $\depth\gr_\mm(R)>0$ and that $K$ is infinite guarantee the existence of a non-zero divisor $x_0\in\gr_\mm(R)$ of degree one. Suppose for a contradiction that there exist $k\ge0$ and an element $f\in(I\mm^{k+1}:\mm)$ which does not belong to $\mm^k$. Then $f\in\mm^i\setminus\mm^{i+1}$ for a unique $i<k$. Thus $\overline{f}=f+\mm^{i+1}\in\gr_\mm(R)_i$ is a non-zero element. Then $\overline{x_0}\overline{f}=x_0f+\mm^{i+2}\in\gr_\mm(I)_{i+1}$ is non-zero as well. Thus $x_0f\notin\mm^{i+2}$ and since $i<k$, then $x_0f\notin\mm^{k+1}$ as well. This is a contradiction, because $x_0\in\mm$ and $f\in(I\mm^{k+1}:\mm)$, thus $x_0f\in I\mm^{k+1}\subseteq\mm^{k+1}$. It follows that $(I\mm^{k+1}:\mm)\subseteq\mm^k$.
        
        By our claim (\ref{eq:claimDao}) and the computation (\ref{eq:computationDao}) we obtain
        $$
        \textup{Soc}_{Q}(\mathcal{R}(\mm)/\mathcal{R}(\mm,I))\ =\ \bigoplus_{k\ge0}\frac{\mm^k\cap(I\mm^{k+1}:\mm)}{I\mm^k}\ =\ \bigoplus_{k\ge0}\frac{(I\mm^{k+1}:\mm)}{I\mm^k}\ =\ \mathfrak{D}_\mm(I).
        $$
        Now, let ${\bf y}:y_1,\dots,y_n$ be a minimal homogeneous system of generators of $\mathcal{R}(\mm)_+$. By fact (iv) and the previous computation, we obtain that
        $$
        	H_{n}({\bf y};\mathcal{R}(\mm)/\mathcal{R}(\mm,I))\ \cong\ e_1\wedge\cdots\wedge e_n\ \mathfrak{D}_\mm(I),
        $$
        and $\mathfrak{d}_1(I)-1=\mathfrak{d}_3(I)-1=\max\{k: \mathfrak{D}_\mm(I)_k\ne0\}\le\reg_{\mathcal{R}(\mm)} \mathcal{R}(\mm)/\mathcal{R}(\mm,I)$. By (v) we have $\reg_{\mathcal{R}(\mm)} \mathcal{R}(\mm)/\mathcal{R}(\mm,I)=\reg_{\mathcal{R}(\mm)}\mathcal{R}(\mm,I)-1$. The assertion follows.
    \end{proof}

    \begin{Corollary}\label{Cor:reg-loc-Dao}
    	Under the same assumptions and notation of Theorem \ref{Thm:DaoInv}, if we suppose in addition that $R$ is regular, then
    	$$
    	\mathfrak{d}_2(I)\le \mathfrak{d}_3(I)=\mathfrak{d}_1(I)\ \le\ \reg_{\mathcal{R}(\mm)}\mathcal{R}(\mm,I).
    	$$
    \end{Corollary}
    \begin{proof}
    	Since $R$ is regular, we have $\gr_\mm(R)\cong K[x_1,\dots,x_n]$, where $n=\dim R$ (see \cite[Proposition 2.2.5]{BH} and \cite[2.2.25(c)]{BH}). Hence $\depth\gr_\mm(R)=n\ge\depth R>0$ and the assertion follows from Theorem \ref{Thm:depth-gr-Dao}.
    \end{proof}

    In particular, the proof of Theorem \ref{Thm:depth-gr-Dao} and Corollary \ref{Cor:reg-loc-Dao} show that
    \begin{Corollary}\label{Cor:Dao-m-Soc}
    	Under the same assumptions and notation of Theorem \ref{Thm:DaoInv}, if we suppose in addition that $\depth\gr_\mm(R)>0$ or that $R$ is regular, then
    	$$
    	\textup{Soc}_{\mathcal{R}(\mm)_+}(\mathcal{R}(\mm)/\mathcal{R}(\mm,I))\ =\ \mathfrak{D}_\mm(I).
    	$$
    \end{Corollary}

    An ideal $J$ is called a \textit{reduction} of $I$ if $JI^k=I^{k+1}$ for some $k\ge0$. Let $J$ be a reduction of $I$. The \textit{reduction number of $I$ with respect to $J$} is defined as the integer
    $$
    \textup{r}_J(I)\ =\ \min\{k\ge0\ :\ JI^k=I^{k+1}\}.
    $$
    
    Let $M$ be a finitely generated graded $R$-module. Then $\reg_RM_{\ge 1}-1\le\reg_RM$. Indeed the short exact sequence $0\rightarrow M_{\ge 1}\rightarrow M\rightarrow M/M_{\ge1}\rightarrow0$ and rule (ii) imply that $\reg_R M_{\ge1}-1\le\max\{\reg_R M-1,\reg_R M/M_{\ge1}\}$. By rule (iii), $\reg_R M/M_{\ge1}=0$. Hence $\reg_R M_{\ge1}-1\le\reg_R M$. We use this fact in the proof of the next result.
    
    \begin{Proposition}\label{Prop:BoundRegReduction-Dao}
    	Let $(R,\mm,K)$ be either a Noetherian local ring or a standard graded Noetherian $K$-algebra with unique homogeneous maximal ideal $\mm$. Let $I\subset R$ be a reduction of $\mm$, which we assume to be homogeneous if $R$ is a $K$-algebra. Then,
    	$$
    	\reg_{\mathcal{R}(\mm)}\mathcal{R}(\mm,I)\ \le\ \reg_{\mathcal{R}(\mm)}\mathcal{R}(\mm).
    	$$
    \end{Proposition}
    \begin{proof}
    	For all $k\ge0$, we have the short exact sequence
    	$$
    	0\rightarrow I\mm^k\rightarrow\mm^{k+1}\rightarrow\frac{\mm^{k+1}}{I\mm^k}\rightarrow0.
    	$$
    	Taking the direct sum of these short exact sequences for all $k\ge0$, we obtain the short exact sequence $0\rightarrow\mathcal{R}(\mm,I)\rightarrow\mathcal{R}(\mm)_{\ge1}(1)\rightarrow\mathcal{R}(\mm)_{\ge1}(1)/\mathcal{R}(\mm,I)\rightarrow0$ of finitely generated graded $\mathcal{R}(\mm)$-modules. Notice that
    	$$
    	\frac{\mathcal{R}(\mm)_{\ge1}(1)}{\mathcal{R}(\mm,I)}\ \ =\ \ \bigoplus_{k\ge0}\frac{\mm^{k+1}}{I\mm^k}\ =\ \bigoplus_{k=0}^{\textup{r}_I(\mm)-1}\frac{\mm^{k+1}}{I\mm^k}.
    	$$
    	Hence, from fact (iii) we have $\reg_{\mathcal{R}(\mm)}\mathcal{R}(\mm)_{\ge1}(1)/\mathcal{R}(\mm,I)=\max\{\textup{r}_I(\mm)-1,0\}$. Now, if $\textup{r}_I(\mm)=0$, then $\mathcal{R}(\mm,I)\cong\mathcal{R}(\mm)_{\ge1}(1)$. Thus, by rule (i) we have that
    	$$
    	\reg_{\mathcal{R}(\mm)}\mathcal{R}(\mm,I)\ =\ \reg_{\mathcal{R}(\mm)}\mathcal{R}(\mm)_{\ge1}(1)\ =\ \reg_{\mathcal{R}(\mm)}\mathcal{R}(\mm)_{\ge1}-1\ \le\ \reg_{\mathcal{R}(\mm)}\mathcal{R}(\mm).
    	$$
    	
    	Suppose now that $\textup{r}_I(\mm)>0$. Then $\reg_{\mathcal{R}(\mm)}\mathcal{R}(\mm)_{\ge1}(1)/\mathcal{R}(\mm,I)=\textup{r}_I(\mm)-1$. By \cite[Theorem 4.9]{NVTrung-1998} we have $\textup{r}_I(\mm)\le\reg_{\mathcal{R}(\mm)}\mathcal{R}(\mm)$. Applying rule (ii) to the above sequence we get
    	\begin{align*}
    		\reg_{\mathcal{R}(\mm)}\mathcal{R}(\mm,I)\ &\le\ \max\{\reg_{\mathcal{R}(\mm)}\mathcal{R}(\mm)_{\ge1}(1),\,\reg_{\mathcal{R}(\mm)}\mathcal{R}(\mm)_{\ge1}(1)/\mathcal{R}(\mm,I)+1\}\\
    		&=\ \max\{\reg_{\mathcal{R}(\mm)}\mathcal{R}(\mm)_{\ge1}-1,\,\textup{r}_I(\mm)\}\\
    		&\le\ \reg_{\mathcal{R}(\mm)}\mathcal{R}(\mm),
    	\end{align*}
        as desired.
    \end{proof}

    Using the theory developed thus far we can reprove \cite[Corollary 3.11]{MNQ23}.
    \begin{Corollary}\label{Cor:MNQ}
    	Under the same assumptions and notation of Theorem \ref{Thm:DaoInv}, if we suppose in addition that $R$ is a regular ring, and that $I$ is a reduction of $\mm$, then $$\mathfrak{d}_1(I)=\mathfrak{d}_2(I)=\mathfrak{d}_3(I)=0.$$
    \end{Corollary}
    \begin{proof}
    	Under the assumption that $R$ is regular, $\mm$ is generated by a regular sequence. Therefore \cite[Corollary 5.2]{NVTrung-1998} implies that $\reg_{\mathcal{R}(\mm)}\mathcal{R}(\mm)=0$. By Proposition \ref{Prop:BoundRegReduction-Dao} and Corollary \ref{Cor:reg-loc-Dao} we conclude that $\mathfrak{d}_1(I)=\mathfrak{d}_2(I)=\mathfrak{d}_3(I)=0$.
    \end{proof}

    We conclude this section by providing a lower bound for $\mathfrak{d}_1(I)=\mathfrak{d}_3(I)$.
    \begin{Proposition}
    	Under the same assumptions and notation of Theorem \ref{Thm:DaoInv}, we have
    	$$
    	\omega(\mathcal{R}(\mm,I):_{\mathcal{R}(R)}\mm)-1\ \le\ \mathfrak{d}_3(I)=\mathfrak{d}_1(I)
    	$$
    \end{Proposition}
    \begin{proof}
    	Let $M=\bigoplus_{k\ge0}M_k$ be a finitely generated module over a standard graded Noetherian algebra $A=\bigoplus_{k\ge0}A_k$ with homogeneous maximal ideal $\nn=(y_1,\dots,y_n)$ and $\deg y_i=1$ for all $i$. Then $M_k=\nn_1 M_{k-1}=\{\sum_iy_if_i:f_i\in M_{k-1}\}$ for all $k\gg0$.
    	
    	Now, let $k\ge\mathfrak{d}_3(I)+1$, then $I\mm^{k-1}$ and $I\mm^k$ are weakly $\mm$-full. Hence
    	$$
    	(I\mm^{k+1}:\mm)\ =\ I\mm^{k}\ =\ \mm(I\mm^{k-1})\ =\ \mm(I\mm^{k}:\mm).
    	$$
    	Thus $(\mathcal{R}(\mm,I):_{\mathcal{R}(\mm)}\mm)_{k+1}= \mathcal{R}(\mm)_1(\mathcal{R}(\mm,I):_{\mathcal{R}(\mm)}\mm)_k$
    	for all $k\ge\mathfrak{d}_3(I)+1$. This shows that $\omega(\mathcal{R}(\mm,I):_{\mathcal{R}(\mm)}\mm)\le\mathfrak{d}_3(I)+1$.
    \end{proof}

	\section{Fullness in Koszul algebras}\label{Sec:Dao2}
	
	In this section, we consider the more tame situation of a Koszul algebra and determine bounds for the Dao numbers of an ideal in such a ring. Let $K$ be a field and let $R=\bigoplus_{k\ge0}R_k$ be a standard graded $K$-algebra. That is,
	\begin{enumerate}
		\item[(i)] $R$ admits a decomposition $R=\bigoplus_{k\ge0}R_k$ as an abelian group
		\item[(ii)] $R$ is generated as an $R_0$-algebra by the finite dimensional $K$-vector space $R_1$,
		\item[(iii)] and $R_0=K$.
	\end{enumerate}
	In this case, we have the surjective homogeneous ring map
	\begin{equation}\label{eq:PresMapDao}
		\varphi\ :\ S=K[x_1,\dots,x_n]\rightarrow R
	\end{equation}
	such that $R_1$ has a $K$-basis given by elements $f_1,\dots,f_n$ of degree one, $\varphi(x_i)=f_i$ for all $i$ and $\varphi(r)=r$ for all $r\in R_0=K$. We call (\ref{eq:PresMapDao}) the \textit{canonical presentation} of $R$. Hence $R\cong S/\ker\varphi$. Let $\mm=R_+=\bigoplus_{k>0}R_k$.
	
	Notice that we may view the field $K$ as a finitely generated graded $R$-module by identifying it with $R/\mm$. Following \cite[Definition 4]{CDNR13}, we say that the algebra $R$ is \textit{Koszul} if $\reg_{R}K=0$. In this case, $\reg_R\mm=1$.
	
	Let $M=\bigoplus_{k\ge0}M_k$ be a finitely generated graded $R$-module. We denote by $M_{\langle k\rangle}$ the submodule of $M$ generated by the $K$-vector space $M_k$.
	
	Now, let $I\subset R$ be a homogeneous ideal. We say that $I$ has \textit{linear resolution} if $\alpha(I)=\omega(I)$ (that is, $I$ is generated in a single degree) and $\reg_RI=\alpha(I)$. We say that $I$ is \textit{componentwise linear} if $I_{\langle j\rangle}$ has linear resolution, for all $j$.\smallskip
	
	\begin{Lemma}\label{Lemma:lir}
		$(I\mm)_{\langle k\rangle}=I_{\langle k-1\rangle}\mm$ for all $k\ge1$.
	\end{Lemma}
	\begin{proof}
		It is clear that $I_{\langle k-1\rangle}\mm$ is contained in $(I\mm)_{\langle k\rangle}$. For the opposite inclusion, observe that the $K$-vector space $(I\mm)_k$ is generated by the non-zero elements of the form $ug$ where $u\in\mm$ is a monomial and $g$ is homogeneous element of $I$ such that $\deg(ug)=k$. Therefore, it is enough to show that if $f=ug$ is as before, then $f\in I_{\langle k-1\rangle}\mm$. We may assume that $x_j$ divides $u$. Then, we can write $f=x_j((u/x_j)g)$, where $(u/x_j)g\in I$ has degree $k-1$. Thus, $f\in I_{\langle k-1\rangle}\mm$.
	\end{proof}
	
	\begin{Proposition}\label{Prop:CLm-fullDao}
		Let $R$ be a Koszul algebra with homogeneous maximal ideal $\mm$, and let $I\subset R$ be a componentwise linear ideal. We assume that $\depth R>0$ and $K$ is infinite. Then the following properties hold:
		\begin{enumerate}
			\item[\textup{(a)}] $I\mm^\ell$ is componentwise linear for any $\ell$.
			\item[\textup{(b)}] $(I\mm:\mm)=I$.
			\item[\textup{(c)}] $\mathfrak{d}_3(I)=0$.
		\end{enumerate}
	\end{Proposition}
	\begin{proof}
		Since $\depth R>0$, we have that $I\mm$ is non-zero.
		
		(a) This fact is well-known. We sketch the proof for completeness. It is enough to prove that $I\mm$ is componentwise linear. We must show that for each $j$ such that $(I\mm)_{\langle j\rangle}=I_{\langle j-1\rangle}\mm$ is non-zero, this ideal has linear resolution. Notice that $\reg_R(I\mm)_{\langle j\rangle}\ge j$. On the other hand, by \cite[Lemma 4]{CDNR13} and since $\reg_RI_{\langle j-1\rangle}=j-1$, we obtain that $\reg_R(I\mm)_{\langle j\rangle}\le\reg_RI_{\langle j-1\rangle}+\reg_R\mm=j$.\smallskip
		
		(b) Firstly, we assume that $I$ has linear resolution, say $d$-linear. It is clear that $I\subseteq(I\mm:\mm)$. Suppose for a contradiction that there exists $f\in(I\mm:\mm)\setminus I$. Then $\overline{f}=f+I\in(I:\mm)/I$, because $f\mm\subseteq I\mm\subseteq I$ and $f\notin I$. By fact (iv) recalled at page \pageref{fact:page}, $H_n({\bf x},R/I)\cong e_1\wedge\dots\wedge e_n\ \textup{Soc}(R/I)\ne0$ where ${\bf x}\ :\ x_1,\dots,x_n$ is a minimal homogeneous system of generators of $\m$. Since $I$ has a $d$-linear resolution, $H_n({\bf x},I)$ is concentrated in degree $n+d-1$. Thus, the $K$-vector space $\textup{Soc}(R/I)$ is generated in a single degree $d-1$, and so $\deg(f)=d-1$. By part (a), $I\mm$ has a $(d+1)$-linear resolution. Hence, the above argument shows that any non-zero homogeneous element of $(I\mm:\mm)\setminus I\mm$ has degree $d$. Since $I\mm$ is generated in degree $d+1$, $(I\mm:\mm)$ is generated in degrees $\ge d$. This contradicts the fact that $\deg(f)=d-1$ and shows that $(I\mm:\mm)=I$.
		
		Now, suppose that $I$ is generated in more than one degree. Notice that the ideal $(I\mm:\mm)$ is homogeneous as well. Thus, by the first part we have
		\begin{align*}
			(I\mm:\mm)\ &=\ \bigoplus_j(I\mm:\mm)_{j}\ =\ \bigoplus_j((I\mm)_{\langle j+1\rangle}:\mm)_{j}\\
			&=\ \bigoplus_j(I_{\langle j\rangle}\mm:\mm)_{j}\ =\ \bigoplus_j(I_{\langle j\rangle})_{j}\\
			&=\ \bigoplus_jI_{j}\ =\ I.
		\end{align*}
		
		(c) It follows by combining (a) with (b).
	\end{proof}
	
	\begin{Corollary}\label{Cor:CLIDao}
		Let $R$ be a Koszul $K$-algebra and let $I\subset R$ be a componentwise linear ideal. Assume that $K$ is infinite and $\depth R>0$. Then $\mathfrak{d}_1(I)=\mathfrak{d}_2(I)=\mathfrak{d}_3(I)=0$. In particular, $I$ is $\mm$-full, full and weakly $\mm$-full.
	\end{Corollary}
	\begin{proof}
		By Proposition \ref{Prop:CLm-fullDao}(c), $\mathfrak{d}_3(I)=0$. Thus \cite[Proposition 2.2]{MNQ23} implies that $\mathfrak{d}_1(I)=\mathfrak{d}_2(I)=0$. The assertion follows.
	\end{proof}
	
	\begin{Proposition}\label{Prop:BoundRegEquiDao}
		Let $R$ be a Koszul algebra with homogeneous maximal ideal $\mm$, and let $I\subset R$ be a homogeneous ideal generated in a single degree. Suppose that $K$ is infinite and $\depth R>0$. Then
		$$
		\mathfrak{d}_2(I)\le \mathfrak{d}_3(I)=\mathfrak{d}_1(I)\ \le\ \reg_{R}I-\alpha(I).
		$$
	\end{Proposition}
	\begin{proof}
		Since $I$ is generated in a single degree, we have $I\mm^k=I_{\langle\alpha(I)+k\rangle}$ for all $k$. From \cite[Proposition 2.8]{CDNR13} it follows that for all $k\ge\reg_RI-\alpha(I)$, the ideal $I_{\langle k\rangle}=I\mm^{k}$ has linear resolution. By Proposition \ref{Prop:CLm-fullDao}, it follows that $I\mm^k$ is weakly $\mm$-full, for all $k\ge\reg_{R}I-\alpha(I)$. Hence $\mathfrak{d}_3(I)\le\reg_{R}I-\alpha(I)$. Since $K$ is infinite and $\depth\,R>0$, \cite[Proposition 2.2]{MNQ23} implies that $\mathfrak{d}_2(I)\le\mathfrak{d}_1(I)=\mathfrak{d}_3(I)\le\reg_{R}I-\alpha(I)$.
	\end{proof}

	The following result generalizes \cite[Theorem 1.1]{FHM23} and the above proposition when $I$ is generated in more than one degree.
	
	\begin{Theorem}\label{Thm:FHM-Dao}
		Let $R$ be a Koszul algebra with maximal ideal $\mm$, and let $I\subset R$ be a homogeneous ideal. Suppose that $K$ is infinite and $\depth R>0$. Then
		$$
		\mathfrak{d}_2(I)\le \mathfrak{d}_3(I)=\mathfrak{d}_1(I)\ \le\ \reg_{\gr_\mm(R)}\gr_\mm(I).
		$$
	\end{Theorem}
	\begin{proof}
		Recall that the \textit{associated graded module} of $I$
		$$
		\gr_\mm(I)\ =\ \bigoplus_{k\ge0}(I\mm^k/I\mm^{k+1})\ =\ (I/\mm I)\oplus (I\mm/I\mm^2)\oplus(I\mm^2/I\mm^3)\oplus\cdots
		$$
		is a finitely generated graded module over the associated graded ring $\gr_\mm(R)$. It follows from \cite[1.5 Proposition (1)-(3)]{HI2005} (see also \cite[Theorem 3.2.8]{R2001}) that a finitely generated graded $R$-module $M$ is componentwise linear if and only if $\gr_\mm(M)$ has a linear resolution. By \cite[Proposition 8]{CDNR13}, $\gr_\mm(I)_{\langle k\rangle}$ has a linear resolution for all $k\ge\reg_{\gr_\mm(R)}\gr_\mm(I)$. Notice that $\gr_\mm(I)_{\langle\ell\rangle}\ =\ \gr_\mm(I\mm^\ell)$ for any integer $\ell$. Thus, $I\mm^k$ is componentwise linear for all $k\ge\reg_{\gr_\mm(R)}\gr_\mm(I)$. By Proposition \ref{Prop:CLm-fullDao}, it follows that $I\mm^k$ is weakly $\mm$-full for all $k\ge\reg_{\gr_\mm(R)}\gr_\mm(I)$. Hence $\mathfrak{d}_3(I)\le\reg_{\gr_\mm(R)}\gr_\mm(I)$. Since $K$ is infinite and $\depth R>0$, by \cite[Proposition 2.2]{MNQ23} we conclude that $\mathfrak{d}_2(I)\le\mathfrak{d}_1(I)=\mathfrak{d}_3(I)\le\reg_{\gr_\mm(R)}\gr_\mm(I)$.
	\end{proof}

    \begin{Question}
    	Let $a\ge b\ge c\ge0$ be non-negative integers. Can we find a graded ideal $I\subset R$ such that $\reg_{\gr_\mm(R)}\gr_\mm(I)=a$, $\mathfrak{d}_1(I)=\mathfrak{d}_3(I)=b$ and  $\mathfrak{d}_2(I)=c$?
    \end{Question}
	
	Next, we present a large class of componentwise linear ideals.
	
	Let $u=x_1^{a_1}\cdots x_n^{a_n}\in S$ be a monomial where $S$ is the polynomial ring in the canonical presentation (\ref{eq:PresMapDao}). Since $R\cong S/\ker\varphi$, if the residue class $\overline{u}=u+\ker\varphi$ is non-zero, then we call $\overline{u}$ a \textit{monomial} of $R$. To simplify the notation, we denote $\overline{u}$ again by $u$. Notice that $u$ may have different representations in $R$. For example, in the Koszul algebra $K[x_1,x_2,x_3,x_4]/(x_1x_2-x_3x_4)$ we have $x_1x_2=x_3x_4$.
	
	Again let $R$ be a Koszul algebra and $I\subset R$ be an ideal. We say that $I$ is a \textit{monomial ideal} of $R$ if $I$ can be generated by monomials of $R$. In such case, we denote by $\mathcal{G}(I)$ any minimal monomial generating set of $I$.
	
	We say that $I$ has \textit{linear quotients} if there exists a minimal monomial generating set $\mathcal{G}(I)$ of $I$ and an order $\mathcal{O}:u_1<\dots<u_m$ of $\mathcal{G}(I)$ such that $(u_1,\dots,u_{j-1}):_R(u_j)$ is generated by variables, for $j=2,\dots,m$.
	
	\begin{Example}\label{Ex:C}
		\rm \cite[Example 1.20]{C2014} Consider the Koszul algebra 
		$$
		R\ =\ \frac{K[a,b,c,d]}{(ac,ad,ab-bd,a^2+bc,b^2)}.
		$$
		In this algebra, the ideal $I=(b)$ does not have linear resolution. Notice that in the usual polynomial ring $S=K[a,b,c,d]$ the ideal $I=(b)$ has linear quotients, and thus a linear resolution \cite[Proposition 8.2.1]{HH2011}.
	\end{Example}
	
	To guarantee that monomial ideals $I\subset R$ with linear quotients are componentwise linear we need to impose further conditions on the algebra $R$. As shown in \cite{KV22}, we must assume that $R$ is \textit{strongly Koszul}. A Koszul algebra $R$ is called \textit{strongly Koszul} if there exists a basis $X$ of $R_1$ such that for every proper subset $Y\subset X$ and every $x\in X\setminus Y$, there exists a subset $Z\subset X$ such that $(Y):_R(x)=(Z)$.
	
	A pivotal property of a strongly Koszul algebra $R$ is that any ideal of $R$ generated by linear forms has linear resolution \cite[Lemma 3.3]{KV22}. Hence, this fact together with \cite[Theorem 3.1]{LZ13} implies that
	\begin{Proposition}\label{Prop:LQ-CL-Dao}
		Let $R$ be a strongly Koszul algebra and $I\subset R$ be a monomial ideal with linear quotients order $\mathcal{O}:u_1<\dots<u_m$ satisfying $\deg(u_1)\le\dots\le\deg(u_m)$. Then $I$ is componentwise linear. In particular, $\mathfrak{d}_3(I)=0$.
	\end{Proposition}
	
	See also \cite{CFL23} for the exterior algebra case. Example \ref{Ex:C} shows that we can not drop the assumption that $R$ is a strongly Koszul algebra.
	
	If $I$ is a monomial ideal in the standard graded polynomial ring $S=K[x_1,\dots,x_n]$ having linear quotients order $\mathcal{O}:u_1<\dots<u_m$, then by the usual Bj\"orner-Wachs rearrangement lemma \cite[Lemma 2.1]{JZ10} (see also \cite{BW96}) we can always also find an order $\mathcal{O}':u_{i_1}<\dots<u_{i_m}$ such that $\deg(u_{i_1})\le\dots\le\deg(u_{i_m})$. Thus, at least in the case of a polynomial ring, the condition on the degrees $\deg(u_1)\le\dots\le\deg(u_m)$ given in Proposition \ref{Prop:LQ-CL-Dao} can be removed. It is not clear whether this is also the case for monomial ideals with linear quotients in a strongly Koszul algebra
	
	\section{The Dao numbers of monomial ideals}\label{Sec:Dao3}
	
	While Theorem \ref{Thm:FHM-Dao} shows that the number $\reg_{\gr_\mm(R)}\gr_\mm(I)$ is an upper bound for the Dao numbers, it may be difficult to compute it. In this section, we consider the more specific class of monomial ideals in a standard graded polynomial ring, and we provide combinatorial bounds for the Dao numbers.\smallskip
	
	Let $S=K[x_1,\dots,x_n]$ be the standard graded polynomial ring over a infinite field $K$ and let $\mm=(x_1,\dots,x_n)$ be the unique homogeneous maximal ideal.
	
	Let ${\bf a}=(a_1,\dots,a_n)\in\ZZ_{\ge0}^n$. We set ${\bf a}[i]=a_i$ for all $i$, ${\bf x^a}=\prod_{i}x_i^{{\bf a}[i]}$. In particular, ${\bf x^0}=1$ for ${\bf 0}=(0,0,\dots,0)$. The monomial ${\bf x^a}$ is called \textit{squarefree} if ${\bf a}[i]\in\{0,1\}$ for all $i$. A monomial ideal $I\subset S$ is called \textit{squarefree} if $I$ is generated by squarefree monomials. The next easy observation will be used several times.
	
	\begin{Remark}\label{Rem:KnotImpDao}
		\rm Let $I\subset S$ be a monomial ideal. Notice that $\depth S=n\ge1$. Since $K$ is infinite, by \cite[Proposition 2.2]{MNQ23} we have $\mathfrak{d}_2(I)\le\mathfrak{d}_1(I)=\mathfrak{d}_3(I)$.
	\end{Remark}
	
	We first combinatorially bound $\reg_{\gr_\mm(S)}\gr_\mm(I)$ for a monomial ideal $I\subset S$. Then, by Theorem \ref{Thm:FHM-Dao} we obtain some bounds for the Dao numbers. For such aim, we recall the next result \cite[Theorem 1.5]{FHM23}.
	\begin{Proposition}\label{Prop:FHM-Dao}
		Let $S=K[x_1,\dots,x_n]$ be the standard graded polynomial ring and let $I\subset S$ be a monomial ideal. Then $I\mm^k$ has linear quotients for all $k\gg0$.
	\end{Proposition}
	
	Analyzing carefully the proof of \cite[Theorem 1.5]{FHM23} we will obtain the desired combinatorial bounds for the Dao numbers.
	
	Let $I\subset S$ be a monomial ideal. In this case, $I$ has a unique minimal monomial generating set, denoted as usual by $\mathcal{G}(I)$. Fix any order $\mathcal{O}: u_1<\cdots<u_m$ of $\mathcal{G}(I)$. For each $2\leq i\leq m$, let $(u_1,\ldots,u_{i-1}):_S(u_i)=(w_{i,1},\ldots,w_{i,\ell_i})$. We set 
	$$
	\lambda_{I,\mathcal{O},u_i}=\sum_{j=1}^{\ell_i}(\deg(w_{i,j})-1),
	$$
	for $2\le i\le m$, and $\lambda_{I,\mathcal{O},u_1}=0$. It is clear that $I$ has linear quotients with respect to the order $\mathcal{O}$ if and only if $\lambda_{I,\mathcal{O},u_i}=0$ for all $1\le i\le r$.
	
	Notice that $I\mm$ is generated by the set $\{x_ju_i\ :\ j=1,\dots,n,\ i=1,\dots,m\}$. After removing the non-minimal generators of $I\mm$ and any repeated element whenever it appears again, we obtain the minimal generating set $\mathcal{G}(I\mm)$. Let $x_{j_1}u_{i_1},x_{j_2}u_{i_2}$ be two minimal generators of $I\mm$. We set $x_{j_1}u_{i_1}<x_{j_2}u_{i_2}$ if $i_1=i_2$ and $j_1<j_2$ or $i_1<i_2$. We denote by $\mathcal{O}_1$ this order. It is shown in the proof of \cite[Theorem 1.5]{FHM23} that for all $i$ and $j$ such that $x_ju_i\in\mathcal{G}(I\mm)$ we have that $\lambda_{I\mm,\mathcal{O}_1,x_ju_i}\leq \lambda_{I,\mathcal{O},u_i}$ and  if $\lambda_{I,\mathcal{O},u_i}>0$, then $\lambda_{I\mm,\mathcal{O}_1,x_ju_i}< \lambda_{I,\mathcal{O},u_i}$. Hence, iterating this process for all $k\ge1$, and calling $\mathcal{O}_k$ the order of $\mathcal{G}(I\mm^k)$ obtained as explained above, we see that for all
	$$
	k\ \ge\ \max\{\lambda_{I,\mathcal{O},u_i}\ :\ i=1,\dots,m\}
	$$
	we have $\lambda_{I\mm^k,\mathcal{O}_k,v}=0$ for all $v\in\mathcal{G}(I\mm^k)$. This means, as we observed above, that $I\mm^k$ indeed has linear quotients with linear quotients order $\mathcal{O}_k$.
	
	Our discussion shows that
	
	\begin{Theorem}\label{Thm:LinQuotBoundDao}
		Let $S=K[x_1,\dots,x_n]$ be the standard graded polynomial ring over an infinite field $K$ and let $I\subset S$ be a monomial ideal. Then
		\[
		\mathfrak{d}_2(I)\le\mathfrak{d}_3(I)=\mathfrak{d}_1(I)\ \le\ \min_{\mathcal{O}}(\max_{u\in\mathcal{G}(I)}\lambda_{I,\mathcal{O},u}),
		\]
		where the minimum is taken over all possible orders $\mathcal{O}$ of $\mathcal{G}(I)$.
	\end{Theorem}
	
	For two monomials $u$ and $v$, we set $u:v=u/\gcd(u,v)$. Let $u_1,\dots,u_j\in S$ be monomials. Then $(u_1,\dots,u_{j-1}):_S(u_j)$ is generated by the monomials $u_i:u_j$ for $i=1,\dots,j-1$ \cite[Proposition 1.2.2]{HH2011}. We denote by $\mu(I)=\dim_K(I/\mm I)$ the minimal number of generators of $I$.
	
	\begin{Corollary}\label{Cor:d1d2d3Dao}
		Let $S=K[x_1,\dots,x_n]$ be the standard graded polynomial ring over an infinite field $K$ and let $I\subset S$ be a monomial ideal. Then
		$$
		\mathfrak{d}_2(I)\le\mathfrak{d}_3(I)=\mathfrak{d}_1(I)\ \le\ (\!\!\,\!\sum_{u\in\mathcal{G}(I)}\deg(u))+1-\mu(I)-\omega(I).
		$$
	\end{Corollary}
	\begin{proof}
		Let $\mathcal{O}:u_1<\dots<u_m$ be an order of $\mathcal{G}(I)$ with $\deg(u_1)\le\dots\le\deg(u_m)$. For all $2\le j\le m$, we have that $(u_1,\dots,u_{j-1}):_S(u_j)=(u_i:u_j, i=1,\dots,j-1)$. Hence, for all $2\le j\le m$, we have
		\begin{align*}
			\lambda_{I,\mathcal{O},u_j}\ &=\ \sum_{i=1}^{j-1}(\deg(u_i:u_j)-1)\ \le\ \sum_{i=1}^{m-1}(\deg(u_i)-1).
		\end{align*}
	
		Since $\sum_{i=1}^{m-1}\deg(u_i)=\sum_{u\in\mathcal{G}(I)}\!\deg(u)-\deg(u_m)=\sum_{u\in\mathcal{G}(I)}\deg(u)-\omega(I)$ and $m=\mu(I)$, we conclude that $\lambda_{I,\mathcal{O},u_j}\le(\sum_{u\in\mathcal{G}(I)}\deg(u))+1-\mu(I)-\omega(I)$, for all $j=1,\dots,m$. The assertion follows from Theorem \ref{Thm:LinQuotBoundDao}.
	\end{proof}
	
	Let $I\subset S$ be a graded ideal. Even if $I$ does not have a unique minimal generating set, all minimal generating sets $f_1,\dots,f_m$ of $I$ have the same number $\mu(I)=\dim_K(I/\mm I)$ of generators and the number $\sum_{i=1}^m\deg(f_i)$ does not depend on the particular minimal generating set. This follows because $I/\mm I$ is a graded $K$-vector space. Therefore, in light of the above corollary, it makes sense to ask
	
	\begin{Question}
		Let $I\subset S$ be a graded ideal. It is true that
		$$
		\mathfrak{d}_2(I)\le\mathfrak{d}_3(I)=\mathfrak{d}_1(I)\ \le\ (\sum_{f\in\mathcal{B}}\deg(f))+1-\mu(I)-\omega(I),
		$$
		where $\mathcal{B}$ is any graded $K$-basis of $I/\mm I$?
	\end{Question}
	
	Now, we propose a different bound for the Dao numbers of monomial ideals.
	
	For this aim, following \cite{F2}, we define the \textit{bounding multidegree} of a monomial ideal $I\subset S$ to be the vector ${\bf deg}(I)=(\deg_{x_1}(I),\dots,\deg_{x_n}(I))$, with
	$$
	\deg_{x_i}(I)\ =\ \max_{u\in\mathcal{G}(I)}\deg_{x_i}(u),\ \ \textup{for all}\ \ \ 1\le i\le n.
	$$
	
	Here for a monomial $u={\bf x^a}$, we set $\deg_{x_i}(u)={\bf a}[i]=\max\{j\ :\ x_i^j\ \textup{divides}\ u\}$. Furthermore, for ${\bf a}\in\ZZ_{\ge0}^n$, we set $|{\bf a}|=\sum_{i=1}^n{\bf a}[i]$. Thus $\deg({\bf x^a})=|{\bf a}|$.
	\begin{Theorem}\label{Thm:GraphsDao}
		Let $I\subset S$ be a monomial ideal. Then
		$$
		\mathfrak{d}_2(I)\le\mathfrak{d}_3(I)=\mathfrak{d}_1(I)\ \le\ \max\{|{\bf deg}(I)|-n,0\}.
		$$
		In particular, if $I$ is squarefree then $\mathfrak{d}_1(I)=\mathfrak{d}_2(I)=\mathfrak{d}_3(I)=0.$
	\end{Theorem}
	
	In order to prove the theorem, we shall need the following lemma. For an integer $n\ge1$, we let $[n]=\{1,2,\dots,n\}$.
	\begin{Lemma}\label{Lem:GraphsDao}
		For all $t\in[n]$ and all $k\ge0$, we have
		$$
		(I\mm^{k+1}:x_t)\ =\ I\mm^{k}+({\bf x^a}/x_t\ :\ {\bf x^a}\in\mathcal{G}(I),\ {\bf a}[t]>0)\mm^{k+1}.
		$$
	\end{Lemma}
	\begin{proof}
		Notice that $\{{\bf x^a}u\ :\ {\bf x^a}\in\mathcal{G}(I),\ u\in\mathcal{G}(\mm^{k+1})\}$ is a generating set of $I\mm^{k+1}$. Hence $\{{\bf x^a}u:x_t\ ,\ {\bf x^a}\in\mathcal{G}(I),\ u\in\mathcal{G}(\mm^{k+1})\}$ is a generating set of $(I\mm^{k+1}:x_t)$. Let ${\bf x^a}\in\mathcal{G}(I)$ and $u\in\mathcal{G}(\mm^{k+1})$. Then,
		$$
		{\bf x^a}u:x_t\ =\ \begin{cases}
			\ ({\bf x^a}/x_t)u&\text{if}\ {\bf a}[t]>0,\\
			\ {\bf x^a}(u/x_t)&\text{if}\ {\bf a}[t]=0,\ x_t\ \textup{divides}\ u,\\
			\ {\bf x^a}u&\textup{if}\ {\bf a}[t]=0,\  x_t\ \textup{does not divide}\ u.
		\end{cases}
		$$
		Since $u\in\mathcal{G}(\mm^{k+1})$ is arbitrary, from the above computations we see that
		$$
		(I\mm^{k+1}:x_t)\ =\ I\mm^{k}\ +\ ({\bf x^a}/x_t\ :\ {\bf x^a}\in\mathcal{G}(I),\ {\bf a}[t]>0)\mm^{k+1}\ +\ I\mm^{k+1}.
		$$
		Finally, from the inclusion $I\mm^{k+1}\subseteq I\mm^{k}$, the asserted formula follows.
	\end{proof}
	
	We are now ready for the proof of the main theorem.
	\begin{proof}[Proof of Theorem \ref{Thm:GraphsDao}]
		By Remark \ref{Rem:KnotImpDao} we have $\mathfrak{d}_2(I)\le\mathfrak{d}_1(I)=\mathfrak{d}_3(I)$ for any monomial ideal $I\subset S$. Suppose that for some $i\in[n]$ the variable $x_i$ does not divide any monomial generator of $I$, then by Lemma \ref{Lem:GraphsDao} we have that $(I\mm^{k+1}:x_i)=I\mm^{k}$ for all $k$. Then, for all $k\ge0$,
		$$
		I\mm^{k}\ \subseteq\ (I\mm^{k+1}:\mm)\ =\ \bigcap_{t=1}^n(I\mm^{k+1}:x_t)\ \subseteq\ (I\mm^{k+1}:x_i)\ =\ I\mm^{k}.
		$$
		Equality follows, and the Dao numbers of $I$ are zero in this case.
		
		Suppose now that for all $i\in[n]$ the variable $x_i$ divides some monomial generator ${\bf x^a}\in\mathcal{G}(I)$. By Lemma \ref{Lem:GraphsDao}, for all $k\ge0$ we have
		\begin{align*}
			(I\mm^{k+1}:\mm)\ &=\ \bigcap_{t=1}^n\,(I\mm^{k+1}:x_t)\\
			&=\ \bigcap_{t=1}^n\big[I\mm^{k}+({\bf x^a}/x_t\ :\ {\bf x^a}\in\mathcal{G}(I),\ {\bf a}[t]>0)\mm^{k+1}\big].
		\end{align*}
		
		Now, let $k\ge|{\bf deg}(I)|-n$. It is clear that $I\mm^{k}$ is contained in $(I\mm^{k+1}:\mm)$. Suppose for a contradiction that there exists a monomial $v\in(I\mm^{k+1}:\mm)$ which does not belong to $I\mm^{k}$. Then, by the above formula, we see that
		\begin{equation}\label{eq:vIntersectDao}
			v\ \in\ \bigcap_{t=1}^n\big[({\bf x^a}/x_t\ :\ {\bf x^a}\in\mathcal{G}(I),\ {\bf a}[t]>0)\mm^{k+1}\big].
		\end{equation}
		
		Let $p\in[n]$ such that $x_p$ divides $v$. Since $J=({\bf x^a}/x_p:{\bf x^a}\in\mathcal{G}(I),{\bf a}[p]>0)\mm^{k+1}$ is non-zero by our assumption and $v\in J$, there exist ${\bf x^a}\in\mathcal{G}(I)$ with ${\bf a}[p]>0$ and a monomial $u\in\mathcal{G}(\mm^{k+1})$ such that $({\bf x^a}/x_p)u$ divides $v$. From this we see that
		$$
		\deg_{x_p}(v)\ge\deg_{x_p}(({\bf x^a}/x_p)u)\ge\deg_{x_p}({\bf x^a}/x_p)=\deg_{x_p}({\bf x^a})-1.
		$$
		
		We claim that $\deg_{x_p}(v)=\deg_{x_p}({\bf x^a})-1$. Suppose that $\deg_{x_p}(v)>\deg_{x_p}({\bf x^a})-1$. Then $x_p^{{\bf a}[p]}$ divides $v$. Now, if $x_p$ divides $u$, then $({\bf x^a}/x_p)(x_p(u/x_p))={\bf x^a}(u/x_p)\in I\mm^{k}$ divides $v$, which would imply that $v\in I\mm^{k}$, against our assumption. Consequently $x_p$ does not divide $u$. Then $x_p({\bf x^a}/x_p)u={\bf x^a}u\in I\mm^{k+1}\subset I\mm^{k}$ divides $v$ which is again a contradiction. Hence, we see that $\deg_{x_p}(v)=\deg_{x_p}({\bf x^a})-1\le\deg_{x_p}(I)-1$. Since $p\in[n]$ is arbitrary, we obtain that
		$$
		\deg(v)\ =\ \sum_{t\in[n]}\deg_{x_t}(v)\ \le\ \sum_{t\in[n]}(\deg_{x_t}(I)-1)\ =\ |{\bf deg}(I)|-n.
		$$
		
		However, by equation (\ref{eq:vIntersectDao}), $v$ must be of degree at least $k+1\ge|{\bf deg}(I)|-n+1$, which is a contradiction. Hence $(I\mm^{k+1}:\mm)=I\mm^{k}$ for all $k\ge|{\bf deg}(I)|-n$, and so $\mathfrak{d}_3(I)\le|{\bf deg}(I)|-n$. Finally, if $I$ is squarefree, then $|{\bf deg}(I)|\le n$. Hence $\max\{|{\bf deg}(I)|-n,0\}=0$ and so $\mathfrak{d}_1(I)=\mathfrak{d}_2(I)=\mathfrak{d}_3(I)=0$.
	\end{proof}
	
	\begin{Examples}
		\rm (a) Let $I=(x^a,y^a)\subset S=K[x,y]$ with $a\ge1$. Then $\mathfrak{d}_i(I)=a-1$ for all $i$ \cite[Example 4.5]{Dao21}. The bound $(\sum_{u\in\mathcal{G}(I)}\deg(u))+1-\mu(I)-\omega(I)=a-1$ given in Corollary \ref{Cor:d1d2d3Dao} is optimal in this case. The bound $|{\bf deg}(I)|-n=2a-2$ provided in Theorem \ref{Thm:GraphsDao} is, however, far from being optimal when $a>1$.
		
		(b) Let $I=(x_1x_2x_3,x_1x_2x_4,x_1x_2x_5,x_1x_3x_4,x_1x_3x_5,x_1x_4x_5)\subset S=K[x_1,\dots,x_5]$. Since $I$ is squarefree, $\mathfrak{d}_i(I)=0$ for $i=1,2,3$. The bound $|{\bf deg}(I)|-n=0$ given in Theorem \ref{Thm:GraphsDao} is optimal, while the bound $(\sum_{u\in\mathcal{G}(I)}\deg(u))+1-\mu(I)-\omega(I)=10$ given in Corollary \ref{Cor:d1d2d3Dao} is not.
		
		(c) For any $d\ge1$, $\mm^d$ is componentwise linear and so $\mathfrak{d}_i(\mm^d)=0$ for all $i$. Notice that $|\mathcal{G}(\mm^d)|=\binom{n+d-1}{d}$ and ${\bf deg}(\mm^d)=(d,d,\dots,d)$. Thus, the bound given in Corollary \ref{Cor:d1d2d3Dao} is $(d-1)\binom{n+d-1}{d}+1$, and the one provided in Theorem \ref{Thm:GraphsDao} is $|{\bf deg}(\mm^d)|-n=(d-1)n$. If $d>1$, then both bounds are not optimal.
	\end{Examples}

    We have the next nice consequence.

	\begin{Corollary}
		Let $I\subset S$ be a monomial ideal. For all $k\ge\max\{|{\bf deg}(I)|-n,0\}$,
		$$
		\textup{Soc}(S/(I\mm^{k+1}))\ =\ \frac{(I\mm^{k+1}:\mm)}{I\mm^{k+1}}\ =\ \frac{I\mm^{k}}{I\mm^{k+1}}.
		$$
		In particular, $\beta_{n-1}(I\mm^{k+1})\ =\ \mu(I\mm^{k})$.
	\end{Corollary}\medskip
	
	Let $u={\bf x^a}\in S$ be a monomial. Its \textit{support} is the set defined as $$\supp(u)\ =\ \{i\ :\ x_i\ \textup{divides}\ u\}\ =\ \{i\ :\ {\bf a}[i]>0\}.$$
	
	Next, we compute the Dao numbers of monomial complete intersections.
	
	We call an ideal $I\subset S$ a \textit{complete intersection} if $I$ is generated by a regular sequence. In particular, it is easy to see that a monomial ideal $I\subset S$ with minimal monomial generating set $\mathcal{G}(I)=\{u_1,\dots,u_m\}$ is a complete intersection if and only if $\supp(u_i)\cap\supp(u_j)\ =\ \emptyset$ for all $1\le i<j\le n$.
	\begin{Corollary}\label{Cor:CIDao}
		Let $I\subset S$ be a monomial complete intersection with minimal monomial generating set $\mathcal{G}(I)=\{u_1,\dots,u_m\}$.
		\begin{enumerate}
			\item[\textup{(a)}] If $\bigcup_{i=1}^m\supp(u_i)\ne[n]$, then $\mathfrak{d}_1(I)=\mathfrak{d}_2(I)=\mathfrak{d}_3(I)=0.$\medskip
			\item[\textup{(b)}] If $\bigcup_{i=1}^m\supp(u_i)=[n]$ and $\mathcal{G}(I)$ contains only pure powers, then $n=m$, up to relabeling $u_i=x_i^{a_i}$ for all $i\in[n]$ with $1\le a_1\le\dots\le a_n$ and $$\mathfrak{d}_1(I)=\mathfrak{d}_3(I)=\sum_{i=1}^{n-1}a_i-(n-1).$$
			\item[\textup{(c)}] If $\bigcup_{i=1}^m\supp(u_i)=[n]$, and $\mathcal{G}(I)$ contains a non-pure power, then $$\mathfrak{d}_1(I)=\mathfrak{d}_2(I)=\mathfrak{d}_3(I)=0.$$
		\end{enumerate}
	\end{Corollary}
	\begin{proof}
		If $\bigcup_{i=1}^m\supp(u_i)\ne[n]$, then the Dao numbers are zero as shown in the first part of the proof of Theorem \ref{Thm:GraphsDao}. Statement (a) follows.
		
		Now, assume that $\bigcup_{i=1}^m\supp(u_i)=[n]$, and $\mathcal{G}(I)$ contains only pure powers. Thus $m=n$ and up to relabeling we can assume that $u_i=x_i^{a_i}$ for all $i\in[n]$ with $1\le a_1\le\dots\le a_n$. By Corollary \ref{Cor:d1d2d3Dao}, we have
		$$
		\mathfrak{d}_2(I)\le\mathfrak{d}_3(I)=\mathfrak{d}_1(I)\le(\!\!\,\!\!\sum_{u\in\mathcal{G}(I)}\deg(u))+1-\mu(I)-\omega(I)=\sum_{i=1}^{n-1}a_i-(n-1).
		$$
		
		Thus, it suffices to show that $I\m^k$ is not weakly $\mm$-full for $k=\sum_{i=1}^{n-1}a_i-n$. For this aim, consider the monomial $v=\prod_{i\in[n]}x_i^{a_i-1}$. Let $j\in[n]$. Notice that
		$$
		x_jv=u_j(\prod_{i\in[n]\setminus\{j\}}x_i^{a_i-1})\ \in\ I\mm^{\sum_{i\in[n]\setminus\{j\}}a_i-(n-1)}.
		$$
		Now, since for all $j\in[n]$ we have
		$$
		\sum_{i\in[n]\setminus\{j\}}\!\!\!a_i-(n-1)\ \ge\ \sum_{i=1}^{n-1}a_i-n+1\ =\ k+1
		$$
		we see that $x_jv\in I\mm^{k+1}$ for all $j\in[n]$. Hence $v\in(I\mm^{k+1}:\mm)$. On the other hand $v\notin I\mm^k$, because for each monomial $w\in I\mm^k\subset I$ there exists $i\in[n]$ such that $\deg_{x_{i}}(w)\ge a_{i}$. This shows that $\mathfrak{d}_1(I)=\mathfrak{d}_3(I)=\sum_{i=1}^{n-1}a_i-(n-1)$ and (b) follows.\smallskip
		
		Finally, assume that $\bigcup_{i=1}^m\supp(u_i)=[n]$ and that $\mathcal{G}(I)$ contains a non-pure power. Up to relabeling, we may assume that $x_1x_2$ divides $u_1\in\mathcal{G}(I)$. Since $I$ is a complete intersection, $x_1$ and $x_2$ do not divide any other generator $u\in\mathcal{G}(I)\setminus\{u_1\}$. Hence, by Lemma \ref{Lem:GraphsDao}, for all $k\ge0$, we have
		\begin{align*}
			(I\mm^{k+1}:x_1)\ &=\ I\mm^{k}+(u_1/x_1)\mm^{k+1},\\
			(I\mm^{k+1}:x_2)\ &=\ I\mm^{k}+(u_1/x_2)\mm^{k+1}.
		\end{align*}
		Assume for a contradiction that $(I\mm^{k+1}:\mm)\ne I\mm^k$ and let $v\in (I\mm^{k+1}:\mm)$ be a monomial not belonging to $I\mm^k$. Arguing as in the proof of Theorem \ref{Thm:GraphsDao}, we see that $v\in(u_1/x_1)\mm^{k+1}\cap(u_1/x_2)\mm^{k+1}$. Thus $v=(u_1/x_1)w_1=(u_1/x_2)w_2$ where $w_1,w_2\in\mm^{k+1}$ are monomials. The previous equation implies that $\deg_{x_1}(v)\ge a_1$. Thus $x_1$ divides $w_1$ and so $v=(u_1/x_1)(x_1(w/x_1))=u_1(w/x_1)\in I\mm^{k}$, a contradiction. Hence, $(I\mm^{k+1}:\mm)=I\mm^k$ for all $k\ge0$, and so $\mathfrak{d}_1(I)=\mathfrak{d}_2(I)=\mathfrak{d}_3(I)=0$.
	\end{proof}\medskip
	
	\textit{Acknowledgment}. The author was partly supported by the Grant JDC2023-051705-I funded by
	MICIU/AEI/10.13039/501100011033 and by the FSE+. The author would like to thank the referee for several useful suggestions, and Cleto B. Miranda-Neto and Douglas Queiroz for their careful reading of the manuscript and several valuable comments.

\end{document}